\documentclass{amsart}
\usepackage{amsmath}
\usepackage{amsfonts}
\usepackage{amssymb}
\usepackage{color}

\newcommand{\ba}{\mathbf{a}} \newcommand{\bx}{\mathbf{x}} 
\newcommand{\bm}{\mathbf{m}}\newcommand{\bn}{\mathbf{n}}\newcommand{\bE}{\mathbf{E}}
\newcommand{\cF}{\mathcal{F}}
\newcommand{\cA}{\mathcal{A}} \newcommand{\cB}{\mathcal{B}}
\newcommand{\cS}{\mathcal{S}} 

\newcommand{\vac}{|0\rangle}

\newtheorem{thm}{Theorem}[section]

\newtheorem{lem}[thm]{Lemma}
\newtheorem{prop}[thm]{Proposition}

\theoremstyle{definition}

\theoremstyle{remark}
\newtheorem{rem}{Remark}[section]

\newcommand{\half}{\frac{1}{2}}

\newcommand{\be}{\begin{equation}}
\newcommand{\ee}{\end{equation}}
\newcommand{\bea}{\begin{eqnarray}}
\newcommand{\eea}{\end{eqnarray}}
\newcommand{\ben}{\begin{eqnarray*}}
\newcommand{\een}{\end{eqnarray*}}
\newcommand{\bt}{\begin{split}}
\newcommand{\et}{\end{split}}
\newcommand{\bet}{\begin{equation}
\begin{split}}
\newcommand{\eet}{\end{split}
\end{equation}}

\DeclareMathOperator{\Span}{span}

\begin{document}

\title{Fermionic gluing principle of the topological vertex}
\date{}
\author{Fusheng Deng \and Jian Zhou}
\address{Fusheng Deng: \ School of Mathematical Sciences, Graduate University of Chinese Academy of Sciences\\ Beijing 100049, China}
\email{fshdeng@gucas.ac.cn}
\address{Jian Zhou: Department of Mathematical Sciences Tsinghua University\\Beijing, 100084, China}
\email{jzhou@math.tsinghua.edu.cn}

\begin{abstract}
We will establish the fermionic gluing principle of the
topological vertex, that is, provided the framed ADKMV conjecture,
the generating functions of the Gromov-Witten invariants of all toric
Calabi-Yau threefolds are Bogoliubov transforms of the vacuum.
\end{abstract}

\maketitle

\section{Introduction}

In general it is an unsolved problem to compute the Gromov-Witten invariants of an algebraic variety
in arbitrary genera.
However,
in the case of toric Calabi-Yau threefolds (which are noncompact),
string theorists have found an algorithm called the topological vertex \cite{AKMV}
to compute the generating function of both open and closed
Gromov-Witten invariants based on a remarkable duality with
 link invariants in the Chern-Simons theory approach of Witten \cite{Wit1, Wit2}.
A mathematical theory of the topological vertex  has been developed in  \cite{LLLZ}.

The topological vertex, which is the generating function of the Gromov-Witten invariants of $\mathbb{C}^3$ with three
special $D$-branes,
is a mysterious combinatorial object that asks for further studies.
On the $A$-theory side, the topological vertex can be
realized as a state in the threefold tensor product of the space $\Lambda$ of symmetric functions.
In this representation its expressions given
by physicists \cite{AKMV} or by mathematicians \cite{LLLZ} are both very complicated.
It is very interesting to understand the topological vertex
from other perspectives.
In \cite{ORV},
the topological vertex is related to a combinatorial problem
of plane partitions.
In \cite{AKMV} it was suggested that the topological vertex is a Bogoliubov
transform (of the vacuum) via the boson-fermion correspondence.
This point of view was further elaborated in \cite{ADKMV} and extended
to the partition functions of toric Calabi-Yau threefolds.
Indeed, by the local mirror symmetry \cite{HV, HIV},
on the B-model side,
one studies quantum Kodaira-Spencer theory of the local mirror curve.
By physical derivations,
the corresponding state is constrained by the Ward identities,
giving the $W_\infty$ constraints.
(See also \cite{Gukov-Sulkowski} where the partition functions
are expected to be annihilated by certain quantum operators
obtaining by quantizing the local mirror curves.)
In this formalism it is natural to use the fermonic picture,
and a simple looking formula (see \S \ref{subsec:framed ADKMV}) for the fermionic form of the topological vertex
under the boson-fermion correspondence was conjectured
in \cite{ADKMV}, which was referred to as the ADKMV conjecture in \cite{DZ}.
The ADKMV conjecture is directly related to integrable hierarchies:
The one-legged case is related to the KP hierarchy, the two-legged case
to the 2-dimensional Toda hierarchy, and the three-legged case to the 3-component KP hierarchy (see Remark \ref{rem:Bogoliubov-KP}).
The one-legged and the two-legged cases can also be seen directly from the bosonic picture \cite{Zh3},
but the three-legged case can only be seen through the fermionic picture.

The topological vertex can be used to compute Gromov-Witten invariants
of toric Calabi-Yau $3$-folds by certain gluing rules.
There is a standard inner product on the space $\Lambda$ of symmetric functions by setting the set of Schur functions
as an orthonormal basis,
and the gluing rule is essentially taking inner product over the components corresponding to
the branes of gluing (see \S \ref{subsec:topological vertex} for exact formulation).
So the resulted generating functions are states in multifold tensor products of the space $\Lambda$.
In general, they
have very complicated combinatorial structures.

In our recent work \cite{DZ}, we proposed a generalization of the  ADKMV conjecture to the framed topological vertex which we refer to as
the framed ADKMV conjecture.
Note that it is important to consider framing
when we consider gluing of the topological vertex. We gave a proof in \cite{DZ} of the framed ADKMV conjecture in the one-legged case and the two-legged case, and derived
a determinantal formula for the framed topological vertex in the three-legged case based on the Framed ADKMV Conjecture.
It remains open to give a proof of this conjecture for the full three-legged topological vertex.

Provided that the framed ADKMV conjecture holds,
then a  natural question is whether or not the generating functions of the Gromov-Witten invariants of general toric Calabi-Yau threefolds are  Bogoliubov transforms
in the fermionic picture.
It was also conjectured in \cite{ADKMV} that it is indeed the case.
However, it seems very difficult to prove this conjecture directly by boson-fermion correspondence and standard Schur calculus,
even for the very simple case of the resolved conifold with a single brane.
In \cite{Sulkowski}
the closed string partition function of the resolved conifold
is related to Hall-Littlewood functions
and a fermionic represenation is obtained by
the deformed boson-fermion correspondence.
Based on the method in \cite{ADKMV},
it was shown in \cite{Kashani-Poor} that the B-model amplitude  of the mirror space
of the one-legged resolved conifold is a Bogoliubov transform,
where how to use the ADKMV conjecture and the gluing rule of the topological
vertex to show this result was also mentioned as an open problem.
It also seems difficult to generalize the method in \cite{ADKMV}
and \cite{Kashani-Poor} to prove this conjecture in general.

In this paper we will tackle this problem using a different strategy.
We will  start from the framed ADKMV conjecture,
and then consider the gluing rule of the topological vertex
as presented in \cite{AKMV}\cite{LLLZ} in the fermionic picture.
Our main aim of this article is to prove
that, provided the framed ADKMV conjecture,
the fermionic form of the generating function of the Gromov-Witten
invariants of any toric Calabi-Yau threefold is a Bogoliubov transform
(see Theorem \ref{thm:ferm. glu. of tv} for exact formulation).
In particular, it is a tau function of multi-component KP hierarchies.
We refer to this result as the fermionic gluing principle of the topological vertex.

By the framed ADKMV conjecture for the framed one-legged and two-legged  topological vertex proved in \cite{DZ},
we  get that the generating functions of the Gromov-Witten invariants of the
total spaces of the bundles $\mathcal{O}(p)\oplus\mathcal{O}(-p-2)\rightarrow \mathbb{P}^1$
with two outer branes on different vertices are  two-component Bogoliubov transforms;
in particular, they are tau functions of the Toda hierarchy.

In fact, we establish the gluing principle of the topological vertex by
proving a gluing principle for general Bogoliubov transforms,
namely, the self-gluing (see \S \ref{sec:self-gluing} for definition) of a Bogoliubov transform
or the gluing (see \S \ref{subsec:gluing} for definition) of two Bogoliubov transforms
is still a Bogoliubov transform.
It may be interesting to generalize our method to prove similar result for general tau functions
of muti-component KP hierarchies which are not necessarily Bogoliubov transforms.

The rest of the paper is arranged as follows. After reviewing some preliminaries in \S 2,
we define the self-gluing of a Bogoliubov transform
and state the self-gluing principle  in \S 3,
and define the gluing of two Bogoliubov transforms
and give the gluing principle  in \S 4.
In \S 5, we apply the results in \S 3 and \S 4 to establish the fermionic
gluing principle of the topological vertex.
In the final \S 6, we give a proof of the self-gluing principle
for Bogoliubov transforms (Theorem \ref{thm:self-gluing rule}).

\vspace{.1in}
{\em Acknowledgements}.
The first author is partially supported by NSFC grants
(11001148 and 10901152) and the President Fund of GUCAS. The second author is partially supported by two NSFC grants (10425101 and 10631050)
and a 973 project grant NKBRPC (2006cB805905).

\section{Preliminaries}\label{sec:prelimilaries}

In this section, we recall briefly some well-known concepts and results that will be used later.

\subsection{Partitions and symmetric functions}\label{subsec:Schur & skew schur}
A partition $\mu$ of a positive integral number $n$ is a decreasing finite sequence of integers $\mu_1\geq\cdots \geq\mu_l>0$,
such that $|\mu| = \mu_1 + \cdots + \mu_l = n$.
The following number associated to $\mu$ will be useful in this paper:
\be
 \kappa_\mu = \sum_{i=1}^l \mu_i(\mu_i - 2i + 1).
\ee
It is very useful to graphically represent a partition by its Young diagram.
This leads to many natural definitions.
First of all,
by transposing the Young diagram one can define the conjugate $\mu^t$ of $\mu$.
Secondly
assume the Young diagram of $\mu$ has $k$ boxes in the diagonal.
Define $m_i = \mu_i - i$ and $n_i = \mu^t_i - i$ for $i = 1, \cdots , k$,
then it is clear that $m_1> \cdots > m_k \geq 0$ and  $n_1> \cdots > n_k \geq 0$.
The partition $\mu$ is completely determined by the numbers $m_i , n_i$.
We often denote the partition $\mu$ by $(m_1, \dots , m_k | n_1, \dots , n_k)$,
this is called the Frobenius notation.
A partition of the form $(m|n)$ in Frobenius form is called a hook partition.

Roughly speaking, a symmetric function is a symmetric polynomial of
infinitely many variables (see \cite{Macdonald} for details).
We denote by $\Lambda$  the space of all symmetric functions in variables $\bx = (x_1, x_2, \dots)$.
For each partition $\mu$, there is an attached  symmetric function $s_\mu$
which is called a Schur function. The Schur function corresponding to the
empty partition is 1.
The inner product on the space $\Lambda$ is defined by setting the set of Schur functions as an orthonormal basis.

Given two partitions $\mu$ and $\nu$,
the skew Schur function $s_{\mu/\nu}$ is defined by the condition
$$(s_{\mu/\nu} , s_\lambda) = (s_\mu , s_\nu s_\lambda)$$
for all partitions $\lambda$. This is equivalent to define
$$s_{\mu/\nu} = \sum_{\lambda}c_{\nu\lambda}^\mu s_\lambda,$$
where the constants $c_{\nu\lambda}^\mu$ are the structure constants
(called the Littlewood-Richardson coefficients) defined by
\be
s_\nu s_\lambda = \sum_{\gamma}c_{\nu\lambda}^\gamma s_\gamma.
\ee

\subsection{Fermionic Fock space }
We say a set of integers $A = \{a_1, a_2, \dots \}\subset \mathbb{Z}+\frac{1}{2}$, $a_1>a_2> \cdots$, is admissible if it satisfies the following two conditions:
\begin{itemize}
\item[1.] $\mathbb{Z}_- + \frac{1}{2}\backslash A$ is finite and
\item[2.] $A\backslash \mathbb{Z}_- + \frac{1}{2}$ is finite,
\end{itemize}
where  $\mathbb{Z}_-$ is the set of negative integers.

Let $W$ be the linear space that is spanned by the basis
$\{\underline{a}| a\in \mathbb{Z}+\frac{1}{2}\}$ indexed by half-integers.
For an admissible set $A = \{a_1, a_2, \dots\}$,
we associate an element $\underline{A}\in \wedge^\infty W$ as follows:
$$\underline{A} = \underline{a_1}\wedge \underline{a_2} \wedge \cdots.$$
Then the free fermionic Fock space $\mathcal{F}$ is defined as
$$\cF = \Span \{\underline{A}: \; A\subset \mathbb{Z}+\frac{1}{2}\; \text{is admissible} \}.$$
We define an inner product on $\mathcal{F}$ by taking
$\{\underline{A}:\; A\subset \mathbb{Z}+\frac{1}{2}\; \text{is admissible} \}$ as an orthonormal basis.

For $\underline{A} = \underline{a_1}\wedge \underline{a_2} \wedge \cdots
\in \mathcal{F}$,
define its charge as:
$$|A\backslash \mathbb{Z}_- + \frac{1}{2}| - |\mathbb{Z}_- + \frac{1}{2}\backslash A|.$$
Denote by $\cF^{(n)} \subset \mathcal{F}$ the subspace spanned by $\underline{A}$ of charge $n$,
then there is a decomposition
$$\mathcal{F} = \bigoplus_{n\in \mathbb{Z}} \cF^{(n)}.$$
An operator on $\mathcal{F}$ is called charge 0 if it preserves the above decomposition.

The charge 0 subspace $\cF^{(0)}$ has a basis indexed by partitions:
\be
|\mu\rangle:= \underline{\mu_1 - \frac{1}{2}} \wedge \underline{\mu_2 - \frac{3}{2}}\wedge \cdots \wedge
 \underline{\mu_l -\frac{2l-1}{2}}\wedge \underline{-\frac{2l+1}{2}}\wedge \cdots
\ee
where $\mu = (\mu_1, \cdots , \mu_l)$,
i.e.,
$|\mu\rangle = \underline{A_\mu}$, where $A_\mu =(\mu_i - i + \half)_{i=1, 2, \dots}$.
If $\mu = (m_1, \cdots , m_k | n_1, \cdots , n_k)$ in Frobenius notation, then
\begin{equation}
|\mu\rangle = \underline{m_1+\frac{1}{2}}\wedge \cdots \wedge \underline{m_k+\frac{1}{2}}\wedge \underline{-\frac{1}{2}}
\wedge \underline{-\frac{3}{2}} \wedge \cdots \wedge \widehat{\underline{-n_k-\frac{1}{2}}} \wedge \cdots \wedge
\widehat{\underline{-n_1-\frac{1}{2}}} \wedge \cdots .
\end{equation}
In particular,
when $\mu$ is the empty partition,
we get:
$$|0\rangle := \underline{-\frac{1}{2}}\wedge \underline{-\frac{3}{2}}\wedge \cdots \in \mathcal{F}.$$
It will be called the fermionic vacuum vector.

We now recall the creators and annihilators on $\mathcal{F}$.
For $r \in \mathbb{Z}+\frac{1}{2}$,
define operators $\psi_r$ and $\psi^*_r$ by
\begin{eqnarray*}
&\psi_r (\underline{A}) =
\begin{cases}
(-1)^{k}\underline{a_1}\wedge\cdots\wedge\underline{a_k}\wedge \underline{r}\wedge\underline{a_{k+1}}\wedge\cdots, & \text{if $a_k > r > a_{k+1}$ for some $k$}, \\
0, &  \text{otherwise};
\end{cases}\\
&\psi^*_r(\underline{A}) =
\begin{cases}
(-1)^{k+1}\underline{a_1}\wedge\cdots\wedge \widehat{\underline{a_k}}\wedge\cdots, & \text{if $a_k = r$ for some $k$}, \\
0, &  \text{otherwise}.
\end{cases}
\end{eqnarray*}
Under the inner product defined above, for $r \in \mathbb{Z}+1/2$, it is clear that $\psi_r $ and $\psi^*_r$ are adjoint operators.
The anti-commutation relations for these operators are
\begin{equation} \label{eqn:CR}
[\psi_r,\psi^*_s]_+:= \psi_r\psi^*_s + \psi^*_s\psi_r = \delta_{r,s}id
\end{equation}
and other anti-commutation relations are zero.
It is clear that for $r > 0$,
\begin{align}
\psi_{-r} \vac & = 0, & \psi_r^* \vac & = 0,
\end{align}
so the operators $\{\psi_{-r}, \psi_r^*\}_{r > 0}$ are called the fermionic annihilators.
For a partition $\mu = (m_1, m_2, . . ., m_k | n_1, n_2, . . ., n_k)$, it is clear that
\be\label{eq:operator rep for mu}
|\mu\rangle = (-1)^{n_1 + n_2 + . . . + n_k}\prod_{i=1}^k \psi_{m_i+\frac{1}{2}} \psi_{-n_i-\frac{1}{2}}^*|0\rangle .
\ee
So the operators $\{\psi_{r}, \psi_{-r}^*\}_{r > 0}$ are called the fermionic creators.
The normally ordered product is defined as
\begin{equation*}
:\psi_r\psi^*_r: =
\begin{cases}
 \psi_r\psi^*_r, & r>0, \\
- \psi^*_r\psi_r, & r<0.
\end{cases}
\end{equation*}
In other words,
an annihilator is always put on the right of a creator.

\subsection{Boson-fermion correspondence}
For any integer $n$, define an operator $\alpha_n$ on the fermionic Fock space $\mathcal{F}$ as follows:
\begin{equation*}
\alpha_n = \sum_{r\in \mathbb{Z} + \frac{1}{2}}:\psi_r\psi^*_{r+n}:.
\end{equation*}
Let
$\mathcal{B} = \Lambda[z , z^{-1}]$
be the bosonic Fock space, where $z$ is a formal variable.
Then the  boson-fermion correspondence is a linear isomorphism
$\Phi: \mathcal{F} \rightarrow \mathcal{B}$ given by
\begin{equation}
u\mapsto z^m \langle\underline{0}_m | e^{\sum_{n=1}^\infty \frac{p_n}{n}\alpha_n}u\rangle ,\ \ u\in \cF^{(m)}
\end{equation}
where $p_n=\sum_{i\geq 1}x^n_i$ are the Newton polynomials and
$|\underline{0}_m\rangle = \underline{-\frac{1}{2}+m}\wedge \underline{-\frac{3}{2}+m}\wedge\cdots$.
It is clear that $\Phi$ induces an isomorphism between $\cF^{(0)}$ and $\Lambda$. Explicitly, this isomorphism is given by
\begin{equation}\label{boson-fermion}
|\mu\rangle \longleftrightarrow s_\mu.
\end{equation}

The boson-fermionic correspondence plays an important role in Kyoto school's theory
of integrable hierarchies.
For example,

\begin{prop}\label{bilinear relation tau fermion}
If $\tau\in \Lambda$ corresponds to $|v\rangle\in F^{(0)}$, then $\tau$ is a $tau$-function of the KP
hierarchy in the Miwa variable $t_n = \frac{p_n}{n}$ if and only if $|v\rangle$ satisfies the bilinear relation
\begin{equation}\label{bilinear relation tau fermion 1}
\sum_{r\in \mathbb{Z} + \frac{1}{2}}\psi_r |v\rangle\otimes \psi^*_r |v\rangle = 0.
\end{equation}
\end{prop}

\begin{rem}\label{rem:Bogoliubov-KP}
A state $|v\rangle\in \cF^{(0)}$ satisfies the bilinear relation
\eqref{bilinear relation tau fermion 1} if and only if it lies in the orbit $\widehat{GL_\infty}|0\rangle$.
This is equivalent to say that $|v\rangle$ can be represented as
$$|v\rangle = \exp(\sum_{r,s\in \mathbb{Z}+1/2}M_{rs}:\psi_{r}\psi_{s}^*:)|0\rangle$$
for some coefficients $M_{rs}$.
There is also a multi-component generalization of the boson-fermion correspondence
which can be used to study multi-component KP hierarchies \cite{KL}.
\end{rem}

\section{Self-gluing principle for Bogoliubov transforms}\label{sec:self-gluing}
In this section, we introduce the notion of Bogoliubov transforms and their self-gluing, and give a statement of the self-gluing principle
of Bogoliubov transforms.

\subsection{Bogoliubov transforms}

On the $N$-component femionic Fock space $\cF_1\otimes\cdots\otimes \cF_N$, where $\cF_1, \ \cdots ,\ \cF_N $ are $N$-copies of $\cF$,
define operators $\psi^i_r$ and $\psi^{i*}_r$, for $r \in {\mathbb Z} + \half$ and $i=1, \cdots , N$.
They act on the $i$-th factor of the tensor product as the operators $\psi_r$ and$\psi_r^*$ respectively,
and we use the Koszul sign convention for the anti-commutation relations for these operators, i.e., we set
\be\label{eqn:sign convention}
[\psi^i_r , \psi^j_s]_+=[\psi^i_r ,\psi^{j*}_s]_+=[\psi^{i*}_r , \psi^{j*}_s]_+ =0
\ee
for $i\neq j$ and $r , s \in {\mathbb Z} + \half$.

For an integer $n$, we denote by $\bn$ the number $n+1/2$.

We call a vector $V\in\cF_1\otimes\cdots\otimes \cF_N$ a Bogoliubov transform (of the vacuum) if it is gotten from the vacuum in $\cF_1\otimes\cdots\otimes \cF_N$  acted upon by an exponential of a quadratic expression of fermionic creators. In other
word, it can be represented as
\be
V = \exp(\sum_{i,j = 1}^N\sum_{m , n \geq 0}A^{ij}_{mn}\psi^i_\bm\psi^{j*}_{-\bn})|0\rangle,
\ee
where $A^{ij}_{mn}$ are certain coefficients possibly with parameters.
Here and in the following,
if not specified otherwise,
for simplicity of notations
we will use $|0\rangle$ to denote
the vacuum $|0\rangle_1\otimes\cdots\otimes|0\rangle_N$ in $\cF_1\otimes\cdots\otimes \cF_N$ and
similar tensor products,
the exact meaning will be clear from the context.

One can see directly form the definition that Bogoliubov transforms are tau functions of multi-component KP hierarchies constructed in
\cite{KL}.

\subsection{The gluing vectors}
Let $a$ and $b$ be two indices, $\cF_a$ and $\cF_b$  two copies of the fermionic Fock space $\cF$.
We call a vector $P^\bE_{ab}\in\cF_a\otimes\cF_b$ of the form
\be\label{eqn:the gluing operator}
P^\bE_{ab} = \exp\left(\sum_{i,j=a,b}\sum_{m,n\geq 0} Q^{m+n+1}
\Theta^{\epsilon^{ab}_{ij}}E^{ij}_{mn}\psi^i_\bm\psi^{j*}_{-\bn}\right)|0\rangle
\ee
a \emph{gluing vector},
where $Q$ , $\Theta$ are formal variables and $\bE = \{E^{ij}_{mn}| m,n \geq 0;\ i,j = a,b\}$ is a series of coefficients maybe with parameters,
and $\epsilon^{ab}_{ij}$ is given by
\begin{eqnarray*}
\epsilon^{ab}_{ij} =
\begin{cases}
1, & \text{if $i = a,\ j = b$}, \\
-1, &  \text{if $i = b,\ j = a$},\\
0, &  \text{if $i = j = a$ or $b$}.
\end{cases}
\end{eqnarray*}
Here $P^\bE_{ab}$ is viewed as a formal power series of $Q$ ,
$\Theta$ and $\Theta^{-1}$ with coefficients in $\cF_a\otimes\cF_b$;
it can also be viewed as a vector in the two-component fermionic Fock space $\cF_a\otimes\cF_b$ with certain parameters.

When $P_{ab}^\bE$ is viewed as a Laurent series in $\Theta$,
write
$$P_{ab}^{\bE} = \sum_{n \in {\mathbb Z} } P_n \Theta^n.$$
Then it is easy to see that
$$P_n \in \cF_a^{(n)} \otimes \cF_b^{(-n)}.$$
In particular,
$$P_0 \in \cF_a^{(0)} \otimes \cF_b^{(0)}.$$

\subsection{Self-gluing of Bogoliubov transforms}\label{subsec:self-gluing definition}
Let $V$ be a Bogoliubov transform in the $(M+2)$-component fermionic Fock space
$\cF_1\otimes\cdots\otimes\cF_M\otimes\cF_a\otimes\cF_b$.
We write
$$V = \exp\left(\sum_{i,j\in\{a, b, 1, 2, \cdots , M\}}\sum_{m , n \geq 0}A^{ij}_{mn}\psi^i_\bm\psi^{j*}_{-\bn}\right)|0\rangle,$$
where $A^{ij}_{mn}$ are certain coefficients maybe with parameters.
There is a natural inner product on $\cF_a\otimes\cF_b$ which is induced form that on $\cF$.
The \emph{self-gluing} $\tilde{G}^\bE(V)$ of $V$ with the gluing vector $P^\bE_{ab}$ is defined to be the inner product
\be\label{eqn:self-gluing}
\tilde{G}^\bE(V) = ( V , P^\bE_{ab})
\in \cF_1\otimes\cdots\otimes\cF_M,
\ee
where the inner product is taken over the components $\cF_a\otimes\cF_b$.
We view $\tilde{G}^\bE(V)$ as a formal power series of $Q$ , $\Theta$ and $\Theta^{-1}$
with coefficients in $\cF_1\otimes\cdots\otimes\cF_M$. The closed part $\tilde{G}^\bE(V)_{closed}$
of the self-gluing is defined to be the inner product
\be
\tilde{G}^\bE(V)_{closed}
= (\exp\left(\sum_{i,j = a, b}\sum_{m , n \geq 0}A^{ij}_{mn}\psi^i_\bm\psi^{j*}_{-\bn}\right)|0\rangle_{ab},
P^\bE_{ab}).
\ee
In the above expressions, $|0\rangle_{ab}$ denotes the vacuum in $\cF_a\otimes\cF_b$.
The quotient
\be
G^\bE(V) : = \tilde{G}^\bE(V)/\tilde{G}^\bE(V)_{closed}
\ee
is called the \emph{normalized self-gluing} $G^\bE(V)$ of $V$ with the gluing vector $P^\bE_{ab}$.
We view $G^\bE(V)$ as a vector in $\cF_1\otimes\cdots\otimes\cF_M$ with parameters,
as well as a formal power series of $Q$ , $\Theta$ and $\Theta^{-1}$ with coefficients in
$\cF_1\otimes\cdots\otimes\cF_M$.

\subsection{The self-gluing principle for Bogoliubov transforms}
One of the main aims of this paper is to prove that the normalized self-gluing of an
arbitrary Bogoliubov transform is again a Bogoliubov transform.
In particular, it is a tau function of  multi-component KP hierarchies.
We refer to this result as the the self-gluing principle for Bogoliubov transforms.

\begin{thm} \label{thm:self-gluing rule}
Let $V$ be a Bogoliubov transform in the $(M+2)$-component fermionic Fock space
$\cF_1\otimes \cdots \otimes \cF_{M}\otimes\cF_{a}\otimes\cF_{b}$, $M>0$.
Then the normalized self-gluing $G^\bE(V)$ of $V$ defined as in \S \ref{subsec:self-gluing definition}
is again a Bogoliubov transform in $\cF_1\otimes\cdots\otimes \cF_{M}$.
\end{thm}

The proof of Theorem \ref{thm:self-gluing rule}
will be presented in \S \ref{sec:proof of self-gluing rule}.

\section{Gluing of two Bogoliubov transforms}
In this section, we apply the result in \S \ref{sec:self-gluing} to study the gluing of two arbitrary Bogoliubov transforms.

\subsection{Gluing principle for Bogoliubov transforms}\label{subsec:gluing}
Let $V_1 \in \cF_1\otimes\cdots\otimes \cF_M \otimes \cF_a$ and $V_2 \in \cF_{M+1}\otimes\cdots\otimes \cF_{M+N} \otimes \cF_{b}$ be two Bogoliubov transforms. We write
\begin{equation*}
\begin{split}
&V_1 = \exp\left(\sum_{i,j\in\{1,\cdots , M, a\}}\sum_{m , n \geq 0}A^{ij}_{mn}\psi^i_\bm\psi^{j*}_{-\bn}\right)|0\rangle,\\
&V_2 = \exp\left(\sum_{i,j\in\{b,M+1,\cdots , M+N\}}\sum_{m , n \geq 0}A^{ij}_{mn}\psi^i_\bm\psi^{j*}_{-\bn}\right)|0\rangle,
\end{split}
\end{equation*}
where $A^{ij}_{mn}$ are certain coefficients maybe with parameters.
Then their tensor product $V_1\otimes V_2$ is a Bogoliubov transform in the
$(M+N+2)$-component fermionic Fock space $\cF_1\otimes\cdots\otimes \cF_{M+N}\otimes \cF_a\otimes \cF_b$.
It has the form
\be
V_1\otimes V_2 = \exp\left(\sum_{i,j\in\{a, b, 1,\cdots , M+N\}}\sum_{m , n \geq 0}A^{ij}_{mn}\psi^i_\bm\psi^{j*}_{-\bn}\right)|0\rangle,
\ee
where we set $A^{ij}_{mn} = 0$ if $i$ and $j$ are not in the same set $\{a, 1, \cdots , M\}$ or $\{b, M+1, \cdots , M+N\}$.
With the gluing vector $P^\bE_{ab}$ defined in \eqref{eqn:the gluing operator},
we define the gluing $\tilde{G}^\bE(V_1 , V_2)$,
the closed part of the gluing $\tilde{G}^\bE(V_1 , V_2)_{closed}$,
and the normalized gluing $G^\bE(V_1 , V_2)$ of $V_1$ and $V_2$
to be $\tilde{G}^\bE(V_1\otimes V_2)$, $\tilde{G}^\bE(V_1\otimes V_2)_{closed}$,
and $G^\bE(V_1\otimes V_2)$ respectively.
Note that we have
$G^\bE(V_1 , V_2) = \tilde{G}^\bE(V_1 , V_2)/\tilde{G}^\bE(V_1 , V_2)_{closed}$.

By the self-gluing principle for Bogoliubov transforms stated in Theorem \ref{thm:self-gluing rule},
we immediately get the gluing principle for two Bogoliubov transforms,
which says that the normalized gluing of two Bogoliubov transforms is again a Bogoliubov transform.

\begin{thm}\label{thm:gluing rule}
Let $V_1 \in \cF_1\otimes\cdots\otimes \cF_M \otimes \cF_a$ and
$V_2 \in \cF_{M+1}\otimes\cdots\otimes \cF_{M+N} \otimes \cF_{b}$ be two Bogoliubov transforms.
Then the normalized gluing $G^\bE(V_1 , V_2)$ of $V_1$ and $V_2$ is a Bogoliubov transform
in  $\cF_1\otimes\cdots\otimes \cF_{M+N}$.
\end{thm}

\begin{rem}
It may be interesting to consider the gluing of two general tau functions of multi-component KP hierarchies.
Let $V_1 \in \cF_1\otimes\cdots\otimes \cF_M \otimes \cF_a$
and $V_2 \in \cF_{M+1}\otimes\cdots\otimes \cF_{N} \otimes \cF_{b}$
that can be represented as
\begin{equation*}
\begin{split}
&V_1 = \exp\left(\sum_{i,j\in\{a,1,\cdots , M\}}\sum_{m , n \in\mathbb{Z}}A^{ij}_{mn}:\psi^i_\bm\psi^{j*}_{\bn}:\right)|0\rangle,\\
&V_2 = \exp\left(\sum_{i,j\in\{b,M+1,\cdots , M+N\}}\sum_{m , n \in\mathbb{Z}}A^{ij}_{mn}:\psi^i_\bm\psi^{j*}_{\bn}:\right)|0\rangle.
\end{split}
\end{equation*}
With the gluing vectors $P^\bE_{ab}$, we can define the gluing  and normalized gluing of
$V_1$ and $V_2$ similarly  as in the case of Bogoliubov transforms.
Motivated by Theorem \ref{thm:gluing rule}, it is natural
to expect that the normalized gluing of $V_1$ and $V_2$ is again a tau function of the $(M+N)$-component KP hierarchy.
The rough meaning here is that the gluing of two KP integrable systems may give us a new KP integrable system.
\end{rem}

\subsection{Specialization of the gluing vectors}\label{subsec:special gluing}
In the context of topological vertex theory, we need to consider a special kind of gluing
of fermionic states.
As we will see, this can be realized as the gluing of Bogoliubov transforms defined in the above subsections,
with a special choice of the gluing vectors.

The gluing vectors we need, denoted by $|P^f_{ab}\rangle$, are vectors
in $\cF_a\otimes\cF_b$ depending on an integer $f$ represented as

\be\label{eqn:special gluing operator}
|P^f_{ab}\rangle = \exp\left[\sum_{n=0}^\infty Q^{n+1/2}(\Theta\epsilon_n\psi_\bn^a\psi_{-\bn}^{b*}+\Theta^{-1}\epsilon'_n
\psi_\bn^b\psi_{-\bn}^{a*})\right]|0\rangle,
\ee
where $Q$ and $\Theta$ are formal variables,  $\epsilon_n , \epsilon'_{n}$ are coefficients with a parameter $q$ that are given by
\be
\begin{split}
&\epsilon_n = i^{1+f}(-1)^{(1+f)n}q^{fn(n+1)/2},\\
&\epsilon'_n = i^{1+f}(-1)^{(1+f)n}q^{-fn(n+1)/2}.
\end{split}
\ee
The vectors $|P^f_{ab}\rangle$ were introduced in \cite{ADKMV} with the motivation of the gluing rule for the topological vertex.
They can be viewed as a formal power series of $Q^{1/2}$ , $\Theta$ and $\Theta^{-1}$ with coefficients in $\cF_a\otimes\cF_b$
with parameters. If we replace $Q$ by $Q^{1/2}$  and set
\begin{eqnarray*}
E^{ij}_{mn} =
\begin{cases}
\delta_{mn}\epsilon_n, & \text{if $i = a,\ j = b,$}, \\
\delta_{mn}\epsilon'_n, &  \text{if $i = b,\ j = a$},\\
\end{cases}
\end{eqnarray*}
in \eqref{eqn:the gluing operator}, then we get $|P^f_{ab}\rangle = P^\bE_{ab}$.\\

When considering the self-gluing and gluing of  Bogoliubov transforms  with respect to the gluing vector $|P^f_{ab}\rangle$, we will denote
$\tilde{G}^\bE(V_1 , V_2)$, $\tilde{G}^\bE(V_1 , V_2)_{closed}$, $G^\bE(V_1 , V_2)$ by $\tilde{G}^f(V_1 , V_2)$,
$\tilde{G}^f(V_1 , V_2)_{closed}$, $G^f(V_1 , V_2)$ respectively, and $\tilde{G}^\bE(V)$, $\tilde{G}^\bE(V)_{closed}$, $G^\bE(V)$ by $\tilde{G}^f(V )$, $\tilde{G}^f(V )_{closed}$, $G^f(V )$ respectively.

A direct corollary of Theorem \ref{thm:self-gluing rule} and Theorem \ref{thm:gluing rule}\label{thm:special gluing rule} is the following

\begin{thm}
1). Let $V\in \cF_1\otimes\cdots\otimes \cF_M \otimes \cF_a\otimes \cF_b$  be a Bogoliubov transform.
Then the normalized self-gluing $G^f(V)$  is a Bogoliubov transform in  $\cF_1\otimes\cdots\otimes \cF_{M}$; \\
2). Let $V_1 \in \cF_1\otimes\cdots\otimes \cF_M \otimes \cF_a$ and $V_2 \in \cF_{M+1}\otimes\cdots\otimes \cF_{M+N} \otimes \cF_{b}$ be two Bogoliubov transforms. Then the normalized gluing $G^f(V_1 , V_2)$  is a Bogoliubov transform in
 $\cF_1\otimes\cdots\otimes \cF_{M+N}$.
\end{thm}

The following lemma,
which was also observed in \cite{ADKMV},
implies that the gluing of Bogoliubov transforms with $|P^f_{ab}\rangle$ as gluing vectors
indeed coincides with the gluing of topological vertex given in \cite{AKMV} and \cite{LLLZ}.

\begin{lem}\label{lem:prop. of the special glu. op.}
Let $\mu_a$ and $\mu_b$ be two partitions, then
\be
\langle\mu_a|\otimes\langle\mu_b|P^f_{ab}\rangle = \delta_{\mu_a\mu_b^t}(-1)^{(f+1)|\mu_a|}q^{f\kappa_{\mu_a}/2}Q^{|\mu|}.
\ee
\end{lem}
\begin{proof}
Let $A_n = Q^{n+1/2}\Theta\epsilon_n , B_n = Q^{n+1/2}\Theta^{-1}\epsilon'_n$. By \eqref{eqn:CR}, we have
$$|P^f_{ab}\rangle =\prod_{n=1}^\infty(1+A_n \psi_\bn^a\psi_{-\bn}^{b*}) \prod_{n=1}^\infty(1+B_n\psi_\bn^b\psi_{-\bn}^{a*}).$$
If we denote by $|P^f_{ab}\rangle_0$ the projection of $|P^f_{ab}\rangle_0$ on the subspace $\cF_a^{(0)}\otimes\cF_b^{(0)}$ with
respect to the decomposition
$$(\mathcal{F}_a\otimes \mathcal{F}_b)^{(0)} = \bigoplus_{n\in \mathbb{Z}}(\mathcal{F}_a^{(n)}\otimes \mathcal{F}_b^{(-n)}),$$
then we have
\begin{equation*}
\begin{split}
|P^f_{ab}\rangle_0& = \sum_{\substack{m_1> \cdots  >m_k\geq 0\\n_1> \cdots  > n_k\geq 0}}\prod_{i=1}^k A_{m_i}B_{n_i}
\psi_{\bm_i}^a\psi_{-\bm_i}^{b*}\psi_{\bn_i}^b\psi_{-\bn_i}^{a*}|0\rangle\\
&=\sum_{\gamma =(m_1, \cdots , m_k |n_1, \cdots , n_k)} \prod_{i=1}^k (-1)^{m_i+n_i}A_{m_i}B_{n_i} |\gamma\rangle\otimes|\gamma^t\rangle\\
&=\sum_\gamma\prod_{i=1}^k Q^{m_i + n_i +1}(-1)^{(1+f)(m_i + n_i +1)}q^{f(m_i(m_i+1)-n_i(n_i+1))/2}|\gamma\rangle\otimes|\gamma^t\rangle.
\end{split}
\end{equation*}
By lemma \ref{lm:kappa}, we see
$$|P^f_{ab}\rangle_0 = \sum_\gamma (-1)^{(1+f)|\gamma|}q^{f\kappa_\gamma/2}Q^{|\gamma|}|\gamma\rangle\otimes|\gamma^t\rangle.$$
By the definition of the inner product on $\cF_a\otimes\cF_b$, we have
$$\langle\mu_a|\otimes\langle\mu_b|P^f_{ab}\rangle = \delta_{\mu_a\mu_b^t}(-1)^{(f+1)|\mu_a|}q^{f\kappa_{\mu_a}/2}Q^{|\mu_a|}.$$
\end{proof}

\section{Fermionic gluing principle of the topological vertex}
 In this section, we apply the results in previous  sections to study gluing of
 the topological vertex in the fermionic picture.
 The main aim here is to establish the fermionic gluing principle of the topological vertex,
 that is, assuming the framed ADKMV conjecture,
 the generating functions of the open Gromov-Witten invariants of all toric
 Clabi-Yau threefolds are Bogoliubov transforms (see Theorem \ref{thm:ferm. glu. of tv} for the exact formulation).

\subsection{The topological vertex and its gluing}\label{subsec:topological vertex}
We give a brief introduction to the theory of topological vertex that introduced in \cite{AKMV},
with a formulation that  fits for our context.

To each toric Calabi-Yau threefold $X$, we can attach a trivalent planar graph $\Gamma$
to it that encodes the loci of degeneration of the $T^2\times \mathbb{R}$ fibcration
of $X$ over $\mathbb{R}^3$. The planar graph $\Gamma$ is called the \emph{toric diagram} of $X$ (see \cite{AKMV} for details).
The toric diagram of $\mathbb{C}^3$ is a trivalent vertex.
Each toric Calabi-Yau threefold can be constructed by gluing $\mathbb{C}^3$ pieces,
which reflects the fact that each toric diagram can be constructed by gluing trivalent vertices.

The generating function of Gromov-Witten invariants of $\mathbb{C}^3$ is given by the topological vertex introduced in \cite{AKMV}
and \cite{LLLZ}. The topological vertex introduced in \cite{AKMV} is defined by

\begin{eqnarray} \label{eqn:TV}
W_{\mu^1, \mu^2, \mu^3}(q)
= \sum_{\rho^1, \rho^3}c_{\rho^1(\rho^3)^t}^{\mu^1(\mu^3)^t}q^{\kappa_{\mu^2}/2+\kappa_{\mu^3}/2}
\frac{W_{(\mu^2)^t\rho^1}(q)W_{\mu^2(\rho^3)^t}(q)}{W_{\mu^2\emptyset}(q)},
\end{eqnarray}
where
$$ c_{\rho^1(\rho^3)^t}^{\mu^1(\mu^3)^t}
= \sum_{\eta} c_{\eta\rho^1}^{\mu^1}c_{\eta(\rho^3)^t}^{(\mu^3)^t},
$$
and the constants $c_{\nu\lambda}^\mu$ are the Littlewood-Richardson coefficients defined in \S \S \ref{subsec:Schur & skew schur}.
It can also be represented in terms of specialization of skew Schur functions as follows (see e.g. \cite{Zh4}):
\be \label{eqn:WSkewSchur}
W_{\mu^1, \mu^2, \mu^3}(q)
= (-1)^{|\mu^2|} q^{\kappa_{\mu^3}/2}
s_{(\mu^2)^t}(q^{-\rho}) \sum_{\eta}
s_{\mu^1/\eta}(q^{(\mu^2)^t+\rho})
s_{(\mu^3)^t/\eta}(q^{\mu^2+\rho}),
\ee
where $q^{\mu+\rho}=(q^{\mu_i-i+1/2})_{i=1,2,\cdots}$.

It is also necessary to consider framings of $\mathbb{C}^3$ and their effects
on the Gromov-Witten invariants.
For the special case $\mathbb{C}^3$, if we index the three edges of
the trivalent vertex by 1, 2 and 3 clockwise,
then the framing of $\mathbb{C}^3$ can be labeled by elements in $\mathbb{Z}^3$.
The topological vertex given by \eqref{eqn:TV} encodes the Gromov-Witten invariants of
$\mathbb{C}^3$ with the canonical framing $(0,0,0)$.
In general, the framed topological vertex with framing $(a_1, a_2, a_3)$,
which encodes the Gromov-Witten invariants of $\mathbb{C}^3$ with
framing $(a_1, a_2, a_3)$, is given by:
\be
W^{(a_1, a_2, a_3)}_{\mu^1, \mu^2, \mu^3}(q) =(-1)^{\sum_{i=1}^3|\mu^i|a_i} q^{\sum_{i=1}^3 a_i\kappa_{\mu^i}/2}
W_{\mu^1, \mu^2, \mu^3}(q).
\ee
The corresponding generating function
\be\label{eqn:framed TV}
Z^{(a_1, a_2, a_3)}(q; \bx^1;\bx^2;\bx^3) = \sum_{\mu^{1},\mu^{2},\mu^{3}}
W^{(a_1,a_2,a_3)}_{\mu^1, \mu^2, \mu^3}(q) s_{\mu^1}(\bx^1) s_{\mu^2}(\bx^2) s_{\mu^3}(\bx^3),
\ee
is a vector in $\Lambda^{\otimes 3}$, where, as before, $\Lambda$ is the space of symmetric functions.

Now we consider general toric Calabi-Yau threefolds.
Let $X$ be a toric Calabi-Yau threefold whose toric diagram is $\Gamma$.
Assume $\Gamma$ has $h$ vertices and $g$ loops.
Then the Gromov-Witten invariants of $X$ with
arbitrary framing can be deduced inductively from the framed topological vertex
by gluing procedure as follows:

$Case\ 1$. Assume $\Gamma$ can be constructed by gluing another toric diagram
$\Gamma'$ and the trivalent vertex along a noncompact edge.
Then the number of vertices of $\Gamma'$ is $h-1$.
Denote by $L_1, \cdots , L_{n-2}, L_a$  the noncompact edges of $\Gamma'$ and by $L_b, L_{n-1}, L_{n}$
the edges of the trivalent vertex ordered clockwise.
Assume we glue  $\Gamma'$ and the  trivalent vertex along $L_a$ and $L_b$.
Assume the generating function of the Gromov-Witten invariants of the toric Calabi-Yau
threefold corresponding to $\Gamma'$ with framing $a_1,\cdots a_{n-2}$ on edges
$L_1, \cdots, L_{n-2}$ and with canonical framing on the edge $L_a$ is given by
$$Z'(\bx^1,\cdots,\bx^{n-2},\bx^a) = \sum_{\mu^1,\cdots,\mu^{n-2}, \mu^a}Z'_{\mu^1,\cdots, \mu^{n-2},\mu^a}s_{\mu^1}(\bx^1)\cdots s_{\mu^{n-2}}(\bx^{n-2})s_{\mu^a}(\bx^a),$$
which is viewed as a vector in $\Lambda^{\otimes (n-1)}$ with certain parameters.
Then the generating function $Z$ of the Gromov-Witten invariants of $X$ with framing
$a_1,\cdots, a_{n}$ on the edges $L_1,\cdots, L_n$ is given by
\begin{equation}\label{eqn:gluing GV}
\begin{split}
&Z (\bx^1,\cdots,\bx^{n})\\
= &\sum_{\mu^1,\cdots,\mu^n,\mu}Z'_{\mu^1,\cdots, \mu^{n-2},\mu}(-1)^{|\mu|} Q^{|\mu|}W^{(f,a_{n-1},a_n)}_{\mu^t, \mu^{n-1}, \mu^{n}} s_{\mu^1}(\bx^1)\cdots s_{\mu^{n}(\bx^n)}\\
=&\sum_{\mu^1,\cdots,\mu^n,\mu}Z'_{\mu^1,\cdots, \mu^{n-2},\mu}(-1)^{(f+1)|\mu|}q^{\kappa_\mu/2} Q^{|\mu|}W^{(0,a_{n-1},a_n)}_{\mu^t, \mu^{n-1}, \mu^{n}}
s_{\mu^1}(\bx^1)\cdots s_{\mu^{n}(\bx^n)},
\end{split}
\end{equation}
where $Q = e^{-t}$ and $t$ is the K\"{a}ler parameter corresponding to the compact edge obtained by gluing $L_a$ and $L_b$.
The framing $f$ on the edge $L_b$ in \eqref{eqn:gluing GV} is determined by
the compatibility  with the canonical framing on the edge $L_a$ (for more details, see \cite{AKMV}).

$Case\ 2$. Assume $\Gamma$ can be constructed from another toric diagram $\Gamma''$
by gluing two noncompact edges.
Then the number of loops in $\Gamma''$ is $g-1$.
Let $L_1,\cdots, L_n, L_a, L_b$ be noncompact edges of $\Gamma''$,
and assume $\Gamma$ is obtained form $\Gamma''$ by gluing the edges $L_a$ and $L_b$.
Assume
\begin{equation*}
\begin{split}
    &Z''(\bx^1,\cdots,\bx^{n},\bx^{a},\bx^{b})\\
  = &\sum_{\mu^1,\cdots,\mu^n,\mu_a,\mu_b}Z''_{\mu^1,\cdots,\mu^n,\mu_a,\mu_b}
s_{\mu^1}(\bx^1)\cdots s_{\mu^n}(\bx^n)s_{\mu^a}(\bx^a)s_{\mu^b}(\bx^b)
\end{split}
\end{equation*}
is the generating function of the Gromov-Witten invariants of the toric Calabi-Yau threefold
$X''$ corresponding to $\Gamma''$ with framing $a_1,\cdots, a_{n}$ on the edges $L_1,\cdots, L_n$
and canonical framing on the two edges $L_a$ and $L_b$.
Then the generating function $Z$ of the Gromov-Witten invariants of $X$ with framing
$a_1,\cdots, a_{n}$ on the edges $L_1,\cdots, L_n$ is given by
\be\label{eqn:self-gluing GV}
\begin{split}
&Z (\bx^1,\cdots,\bx^{n})\\
= &\sum_{\mu^1,\cdots,\mu^n,\mu}Z''_{\mu^1,\cdots, \mu^{n},\mu,\mu^t}(-1)^{(f+1)|\mu|} Q^{|\mu|}q^{-f\kappa_u/2}s_{\mu^1}(\bx^1)\cdots s_{\mu^{n}(\bx^n)},
\end{split}
\ee
where $Q = e^{-t}$ as above and $t$ is the K\"{a}ler parameter corresponding to the
compact edge obtained by gluing $L_a$ and $L_b$,
and we also require that the framing $f$ on the edge $L_b$ matches with the canonical framing on $L_a$ .

In conclusion, we can compute
by induction (with respect to the numbers of vertices and loops of the toric diagrams)
the Gromov-Witten invariants of all toric Calabi-Yau threefolds with arbitrary framing,
with the framed topological vertex as the initial datum.
In this  process, two K\"ahler parameters corresponding to two compact edges are identified
if the two edges give rise to two homologous copies of $\mathbb{P}^1$ in $X$.

\subsection{Fermionic representation of the framed topological vertex}\label{subsec:framed ADKMV}

The framed topological vertex given in \eqref{eqn:framed TV} is a vector in the space $\Lambda^{\otimes 3}$,
where $\Lambda$ is the space of symmetric functions.
By the boson-fermion correspondence,
it corresponds to a vector  in $\cF^{(0)}_1\otimes\cF^{(0)}_2\otimes\cF^{(0)}_3\subset \cF_1\otimes\cF_2\otimes\cF_3$,
where $\cF_i, i=1, 2,3$ are three copies of the fernionic Fock space $\cF$, and $\cF^{(0)}_i$
is the charge 0 subspace of $\cF_i.$
Though the topological vertex given in \eqref{eqn:TV} has an extremely complicated expression
in terms of symmetric functions,
it is conjectured in \cite{AKMV} and \cite{ADKMV} that it has a simple expression in the fermionic picture.
This conjecture was referred to as the ADKMV conjecture in \cite{DZ}.
The ADKMV Conjecture states that, for arbitrary partitions $\mu^1,\ \mu^2,\ \mu^3,$
\be
W_{\mu^1, \mu^2, \mu^3}(q)
= \langle \mu^1|\otimes\langle \mu^2|\otimes\langle \mu^3| \exp \biggl( \sum_{i,j=1}^3\sum_{m,n\geq 0}
A_{mn}^{ij}(q) \psi^{i}_{\bm}\psi^{j*}_{-\bn} \biggr) \vac \otimes \vac \otimes \vac,
\ee
where for $i=1,2,3$, the coefficients are given by
\begin{eqnarray*}
&& A_{mn}^{ii}(q)  = (-1)^n  \frac{q^{m(m+1)/4-n(n+1)/4}}{[m+n+1][m]![n]!}, \\
&& A_{mn}^{i(i+1)}(q)  = (-1)^n q^{m(m+1)/4-n(n+1)/4+1/6}
\sum_{l=0}^{\min(m,n)} \frac{q^{(l+1)(m+n-l)/2}}{[m-l]![n-l]!}, \\
&& A_{mn}^{i(i-1)}(q)  = (-1)^{n+1} q^{-m(m+1)/4+n(n+1)/4-1/6}
\sum_{l=0}^{\min(m,n)} \frac{q^{-(l+1)(m+n-l)/2}}{[m-l]![n-l]!}.
\end{eqnarray*}
Here it is understood that $A^{34} = A^{30}$ and $A^{10} =A^{13}$.

The ADKMV conjecture was generalized to the framed topological vertex in \cite{DZ} as:
\be \label{eqn:FramedADKMV}
W^{(\ba)}_{\mu^1, \mu^2, \mu^3}(q)
= \langle \mu^1|\otimes\langle \mu^2|\otimes\langle \mu^3| \exp  \biggl( \sum_{i,j=1}^3\sum_{m,n\geq 0}
A_{mn}^{ij}(q;\ba) \psi^{i}_{\bm}\psi^{j*}_{-\bn} \biggr)  \vac \otimes \vac \otimes \vac
\ee
for $A^{ij}_{mn}(q;\ba)$ similar to $A^{ij}_{mn}(q)$ above:
\begin{equation*}
\begin{split}
A^{ii}_{mn}(q;\ba)
       =& (-1)^{(m+n+1)a_i+n} q^{(2a_i+1)(m(m+1)-n(n+1))/4}\frac{1}{[m+n+1][m]![n]!}, \\
A_{mn}^{i(i+1)}(q;\ba) =& (-1)^{ma_i+(n+1)a_{i+1}+n} q^{\frac{(2a_i+1)m(m+1)-(2a_{i+1}+1)n(n+1)}{4}+\frac{1}{6}}\\
 &\ \ \ \ \ \sum_{l=0}^{\min(m , n)}\frac{q^{\frac{1}{2}(l+1)(m+n-l)}}{[m-l]![n-l]!},\\
A_{mn}^{i(i-1)}(q;\ba) = &-(-1)^{ma_i+(n+1)a_{i-1}+n} q^{\frac{(2a_i+1)m(m+1)-(2a_{i-1}+1)n(n+1)}{4}-\frac{1}{6}}\\
&\ \ \ \ \sum_{l=0}^{\min(m , n)}\frac{q^{-\frac{1}{2}(l+1)(m+n-l)}}{[m-l]![n-l]!}.
 \end{split}
\end{equation*}
Here $\ba = a_1, a_2, a_3$.
We refer to this conjecture as the Framed ADKMV Conjecture.
The Framed ADKMV Conjecture for the one-legged and two-legged framed topological vertex was proved in \cite{DZ}.
It remains open to give a proof of the (framed) ADKMV conjecture for the full three-legged topological vertex.

A straightforward application of the ADKMV Conjecture and the Framed ADKMV Conjecture is that
they establish a connection between the topological vertex and integrable hierarchies as pointed out in \cite{ADKMV}.
\subsection{Fermionic gluing principle of the topological vertex - simple cases}
 Let $V^{(a_1,a_2,a_3)}_0\in \cF^{(0)}_1\otimes\cF^{(0)}_2\otimes\cF^{(0)}_3$ be the vector corresponding to the framed topological vertex given by \eqref{eqn:framed TV} under the boson-fermion correspondence. The framed ADKMV conjecture stated in the above subsection implies that there is a Bogoliubov transform $V^{(a_1,a_2,a_3)} \in \cF_1\otimes\cF_2\otimes\cF_3$ such that its projection to the subspace $\cF^{(0)}_1\otimes\cF^{(0)}_2\otimes\cF^{(0)}_3$ with respect to the charge decomposition coincides with $V^{(a_1,a_2,a_3)}_0$.

To show clearly how the fermionic gluing principle works,
we first consider two simple examples,
with the framed ADKMV conjecture as the starting point.

The first example is the total space $X_p$ of the bundle
$\mathcal{O}(p)\oplus\mathcal{O}(-p-2)\rightarrow \mathbb{P}^1, p\in\mathbb{Z}.$
The toric diagram $\Gamma_{X_p}$ of $X_p$ contains two vertices,
which reflects that $X_p$ can be constructed by gluing two $\mathbb{C}^3$ pieces.
Let $L_1,\cdots, L_4$ be the noncompact edges of $\Gamma_{X_p}$, with $L_1$, $L_2$
in one vertex and $L_3$, $L_4$ in the other.
Let $\tilde{Z}$ be the generating function of the Gromov-Witten invariants of $X_p$,
and let $Z_{closed}$ be the generating function of the closed Gromov-Witten invariants
(i.e., the Gromov-Witten invariants with no branes) of $X_p$.
Then $Z_{X_p}:=\tilde{Z}/Z_{closed}$ is the generating function of the open
Gromov-Witten invariants of $X_p$.
Under the boson-fermion correspondence (a multi-component version of \eqref{boson-fermion}),
$Z_{X_p}$ corresponds to a vector, say $V_{X_p}$, in $\cF_1^{(0)}\otimes\cdots\otimes\cF_4^{(0)}$.

For an integer $f$, provided the framed ADKMV conjecture, we can define the normalized gluing
$V=G^f(V^{(a_1, a_2, 0)}, V^{(0,a_3, a_4)})$
of $V^{(a_1, a_2, 0)}$ and $V^{(0,a_3, a_4)}$
with the special gluing vector $|P^f_{ab}\rangle$ defined as in \S \ref{subsec:special gluing},
where we label the edges of the first vertex by $L_1$, $L_2$, $L_a$,
and label those of the second vertex by $L_b$, $L_3$, $L_4$.
By Theorem \ref{thm:special gluing rule}, $V$ is a Bogoliubov transform in $\cF_1\otimes\cdots\otimes\cF_4$.
We assume $\Gamma_{X_p}$ is constructed by gluing the two vertices along $L_a$ and $L_b$.
By Lemma \ref{lem:prop. of the special glu. op.} and the gluing rule of the topological vertex
given in \eqref{eqn:gluing GV}, for a suitable choice of $f$ (depends on $p$),
the projection of $V$ on $\cF^{(0)}_1\otimes\cdots\otimes\cF^{(0)}_4$
coincides with $V_{X_p}$.
In other words, if we write
$$Z_{X_p}(\bx^1,\cdots,\bx^4)=\sum_{\mu_1, \cdots, \mu_4}Z_{\mu_1 \cdots \mu_4}s_{\mu_1}(\bx^1)\cdots s_{\mu_4}(\bx^4),$$
then we have $Z_{\mu_1 \cdots \mu_4}=\langle\mu_1|\otimes\cdots\otimes\langle\mu_4|V\rangle$,
for all partitions $\mu_1, \cdots, \mu_4$.

Combine Theorem \ref{thm:ferm. glu. of tv} and the solution in \cite{DZ} of  the famed ADKMV
conjecture for  the cases of the one-legged and the two-legged framed topological vertex , we directly get
\begin{thm}\label{thm:resolved conifold}
Let $X_p$ be the total space of the bundle $\mathcal{O}(p)\otimes\mathcal{O}(-p-2)\rightarrow\mathbb{P}^1$, $p\in\mathbb{Z}$.
Let
$$Z_{X_p}(\bx , \mathbf{y}) = \sum_{\mu , \nu}Z_{\mu,\nu}s_\mu(\bx)s_\mu(\mathbf{y})$$
be the generating  function of the open Gromov-Witten invariants of $X$ with two outer branes  on different vertices.
Then there is a Bogoliubov transform $V\in \cF_1\otimes\cF_2$ such that
$$Z_{\mu,\nu} = \langle\mu|\otimes\langle\nu|V\rangle$$
for all partitions $\mu, \nu$.
\end{thm}

The second example is the local $\mathbb{P}^2$, i.e., the
total space of the vector bundle $\mathcal{O}(-3)\rightarrow \mathbb{P}^2$.
We denote it by $X$. Its toric diagram $\Gamma_X$ contains three vertices
and a loop. So $\Gamma_X$ can be constructed by gluing two noncompact edges of
a toric diagram, say $\Gamma'$, with three vertices and no loops.
Let $X'$ be the toric Calabi-Yau threefold corresponding to $\Gamma'$.
Repeat the process as above, one can show that
there is a Bogoliubov transform $V'\in\cF_1\otimes\cF_2\otimes\cF_3\otimes\cF_a\otimes\cF_b$
whose projection on $\cF^{(0)}_1\otimes\cF^{(0)}_2\otimes\cF^{(0)}_3\otimes\cF^{(0)}_a\otimes\cF^{(0)}_b$
corresponds, under boson-fermion correspondence, to the generating function
of the open Gromov-Witten invariants of $X'$ with framings.
Here we label the noncompact edges of $\Gamma'$ by
$L_1, L_2, L_3, L_a$ and $L_b$ such that $\Gamma_X$ is constructed
form $\Gamma'$ by gluing $L_a$ and $L_b$.
Provided the framed ADKMV conjecture, we can define the normalized
self-gluing $G^f(V')$ of $V'$ with the special gluing vector $|P^f_{ab}\rangle$.
By Theorem \ref{thm:special gluing rule}, $G^f(V')$ is a Bogoliubov transform
in $\cF_1\otimes\cF_2\otimes\cF_3$.
Unlike the situation in the first example, the projection of $V$
on $\cF^{(0)}_1\otimes\cF^{(0)}_2\otimes\cF^{(0)}_3$, denoted by $V_0$, contains a new formal
variable $\Theta$ that doesn't appear in $Z_X$, the generating function of
the open Gromov-Witten invariants of $X$. On the other hand, also by Lemma \ref{lem:prop. of the special glu. op.}
and the gluing rule of the topological vertex given in \eqref{eqn:self-gluing GV}, we see
that the constant term (with respect to $\Theta$) corresponds to $Z_X$ under the boson-fermion correspondence.
In other words, for any partitions $\mu_i$, $i=1,2,3$, the open Gromov-Witten invariants
$Z_{\mu_1,\mu_2,\mu_3}$ of $X$ with boundary conditions given by $\mu_1, \mu_2$ and $\mu_3$ is given by
$$Z_{\mu_1,\mu_2,\mu_3}=\frac{1}{2\pi i}\oint \frac{\langle\mu_1|\otimes\langle\mu_2|\otimes\langle\mu_3|V\rangle}{\Theta}.$$

\subsection{Fermionic gluing principle of the topological vertex - general cases}
We start by assuming the framed ADKMV conjecture is true, and
let $V^{(a, b,c)}$ $(a,b,c\in\mathbb{Z})$ be the vectors defined as in the first paragraph of the above subsection.

Let $X$ be a toric Calabi-Yau threefold whose toric diagram is $\Gamma_X$.
Assume $\Gamma_X$ has $g$ loops and $h$ vertices and $n$ noncompact edges $L_1, \cdots, L_n$.
Let  $\tilde{Z}$ be the generating function of
the Gromov-Witten  invariants of $X$ with framings $a_1, \cdots, a_n$ on $L_1, \cdots, L_n$ respectively.
Denote by $Z_{closed}$ the generating function of the closed Gromov-Witten invariants of $X$.
Then $Z_X = \tilde{Z}/Z_{closed}$ is the generating function of the open Gromov-Witten invariants of $X$.
As formulated in \S \ref{subsec:topological vertex}, $Z_X$ is a vector in $\Lambda^{\otimes n}$ with parameters, where
$\Lambda$ is the space of symmetric functions.
By the boson-fermion correspondence (multi-component version of \eqref{boson-fermion}),
$Z_X$ corresponds to a vector, say $V_X$, in $\cF^{(0)}_1\otimes\cdots\otimes\cF^{(0)}_n$.

There are two possible cases that are similar to those described in \S \ref{subsec:topological vertex}:

$Case\ 1.$ There exists a  toric Calabi-Yau threefold $X'$ with toric diagram $\Gamma_{X'}$
such that $\Gamma_X$ can be constructed by gluing $\Gamma_{X'}$ and the trivalent vertex along a noncompact edge.
Then $\Gamma_{X'}$ has $g$ loops and $h-1$ vertices.
Let $L_1,\cdots ,L_{n-2},L_a$ be the noncompact edges of $\Gamma_{X'}$,
and let $Z_{X'}$ be the generating function of the open Gromov-Witten invariants of $X'$,
with framing $a_1,\cdots, a_{n-2}$ on $L_1, \cdots, L_{n-2}$ respectively and with canonical
framing on $L_a$.
Assume that there is a Bogoliubov transform $V_{X'}\in\cF_1\otimes\cdots\otimes\cF_{n-2}\otimes\cF_a$
whose projection $V_{X',0}$ on $\cF^{(0)}_1\otimes\cdots\otimes\cF^{(0)}_{n-2}\otimes\cF^{(0)}_a$
is a Laurent polynomial of $g$ formal variables $\Theta_1, \cdots, \Theta_g$.
Assume the constant term (with respect to $\Theta_1, \cdots, \Theta_g$) of $V_{X',0}$ gives
the image of $Z_{X'}$ under the boson-fermion correspondence.
Provided the framed ADKMV conjecture,
we can define the normalized gluing
$$V'= G^f(V_{X'} , V^{(0,a_{n-1},a_n)})$$
of $V_{X'}$ and $V^{(0,a_{n-1},a_n)}$
with the special gluing vector $|P^f_{ab}\rangle$ (see \S \ref{subsec:special gluing}),
where $f$ is an integer, and we label the edges of the trivalent vertex by $L_b, L_{n-1}, L_n$.
By Theorem \ref{thm:special gluing rule}, $V'$ is a Bogoliubov transform
in $\cF_1\otimes\cdots\otimes\cF_{n}$.
By Lemma \ref{lem:prop. of the special glu. op.} and the gluing procedure as shown in \eqref{eqn:gluing GV},
if $\Gamma_X$ can be constructed by gluing $\Gamma_{X'}$ and the trivalent vertex along the
noncompact edges $L_a$ and $L_b$,
then, for a suitable choice of $f$, the constant term (with respect to $\Theta_1, \cdots, \Theta_g$)
of the projection of $V'$ on $\cF^{(0)}_1\otimes\cdots\otimes\cF^{(0)}_n$
coincides with $V_X$.

$Case\ 2.$ There exists a toric Calabi-Yau threefold $X''$ whose toric diagram $\Gamma_{X''}$
has $n+2$ noncompact edges, say $L_1,\cdots ,L_{n},L_a,L_b$, such that
$\Gamma_X$ can be constructed from $\Gamma_{X''}$ by gluing $L_a$ and $L_b$.
Then $\Gamma_{X''}$ has $g-1$ loops and $h$ vertices.
Let $Z_{X''}$ be the generating function of the open Gromov-Witten invariants of $X''$,
with framings $a_1,\cdots a_{n}$ on $L_1, \cdots, L_{n}$ respectively and with canonical
framing on $L_a$ and $L_b$.
Assume that there is a Bogoliubov transform $V_{X''}\in\cF_1\otimes\cdots\otimes\cF_{n}\otimes\cF_a\otimes\cF_b$
whose projection $V_{X'',0}$ to $\cF^{(0)}_1\otimes\cdots\otimes\cF^{(0)}_{n}\otimes\cF^{(0)}_a\otimes\cF_b^{(0)}$
is a Laurent polynomial of $g-1$ formal variables $\Theta_1, \cdots, \Theta_{g-1}$.
Assume the constant term (with respect to $\Theta_1, \cdots, \Theta_{g-1}$) of $V_{X'',0}$ gives
the image of $Z_{X''}$ under the boson-fermion correspondence.
Provided the framed ADKMV conjecture, we can define the normalized self-gluing
$$V''= G^f(V_{X''})$$
of $V_{X''}$ with the special gluing vector $\langle P^f_{ab}|$,
(see \S \ref{subsec:special gluing}),
where $f$ is an integer.
By Theorem \ref{thm:special gluing rule}, $V''$ is a Bogoliubov transform in $\cF_1\otimes\cdots\otimes\cF_n$.
Note that $V''$ contains a new formal variable, say $\Theta_g$,
that still appears in the projection $V''_0$ of $V''$ on $\cF_1^{(0)}\otimes\cdots\otimes\cF_n^{(0)}$.
By Lemma \ref{lem:prop. of the special glu. op.} and the gluing procedure as shown in \eqref{eqn:self-gluing GV},
there exists a suitable integer $f$ such that
the constant term (with respect to $\Theta_1,\cdots, \Theta_g$) of
the projection of $V''$ to $\cF^{(0)}_1\otimes\cdots\otimes\cF^{(0)}_n$
gives $V_X$.

By the gluing procedure of the topological vertex as shown in \S\ref{subsec:topological vertex},
the above discussion implies:
\begin{thm}\label{thm:ferm. glu. of tv}
Let $X$ be a toric Calabi-Yau threefold with $n$ outer branes.
 Assume the generating function of the open Gromov-Witten invariants of $X$ is
$$Z(\bx^1, \cdots, \bx^n) = \sum_{\mu^1, \cdots, \mu^n}Z_{\mu^1 \cdots \mu^n} s_{\mu^1}(\bx^1)\cdots s_{\mu^n}(\bx^n).$$
Provided the framed ADKMV conjecture, we have:\\
1). if the toric diagram of $X$ has no loops,
then there is a Bogoliubov transform $V\in\cF_1\otimes\cdots\otimes\cF_{n}$ such that
$$Z_{\mu^1 \cdots \mu^n} = \langle\mu^1|\otimes \cdots\otimes \langle\mu^n|V\rangle$$
for all partitions $\mu^1, \cdots, \mu^n$; and\\
2). if the toric diagram of $X$ has $g$ loops, then there is a Bogoliubov transform $V\in\cF_1\otimes\cdots\otimes\cF_{n}$,
which is a Laurent series of $g$ formal parameters $\Theta_1,\cdots,\Theta_g$, such that
$$Z_{\mu^1 \cdots \mu^n} = \frac{1}{(2\pi i)^g}\oint \frac{\langle\mu^1|\otimes \cdots\otimes \langle\mu^n|V\rangle}{\Theta_1\cdots\Theta_g}d\Theta_1\cdots d\Theta_g$$
for all partitions $\mu^1, \cdots, \mu^n$.
\end{thm}

We refer to Theorem \ref{thm:ferm. glu. of tv} as the fermionic gluing principle of the topological vertex.

\begin{rem}
Let $V\in \cF_1\otimes\cdots\otimes\cF_{n}$ be given by 2) in Theorem \ref{thm:ferm. glu. of tv}.
Let $$Z(\Theta_1,\cdots,\Theta_g)=\sum_{\mu^1, \cdots, \mu^n}\langle\mu^1|\otimes \cdots\otimes \langle\mu^n|V\rangle s_{\mu^1}(\bx^1)\cdots s_{\mu^n}(\bx^n),$$
which is a Laurent series of  formal parameters $\Theta_1,\cdots,\Theta_g$ with coefficients in $\Lambda^{\otimes n}$.
We can view $Z(\Theta_1,\cdots,\Theta_g)$ as an extended topological string partition function of $X$.
Expanding $Z(\Theta_1,\cdots,\Theta_g)$ as
$$Z(\Theta_1,\cdots,\Theta_g)=\sum_{N_1,\cdots,N_g\in \mathbb{Z}}Z_{N_1,\cdots,N_g}\Theta^{N_1}_1\cdots \Theta^{N_g}_g,$$
we have seen that the constant term $Z_{0,\cdots, 0}$ gives the generating function of the open Gromov-Witten invariants of $X$,
the toric Calabi-Yau space considered in 2) in Theorem \ref{thm:ferm. glu. of tv}.
It is natural to consider the geometric meaning of other terms.
It was suggested in \cite{ADKMV} and \cite{ACDKV} from the B-theory viewpoint that $Z_{N_1,\cdots, N_g}$
give the partition functions of the Kodaira-Spencer fields on the mirror curve of $X$ with certain
monodromy datum represented by $(N_1,\cdots,N_g)$.
Surprisingly, it was also argued in the same papers that the additional sectors $Z_{N_1,\cdots, N_g}$ contain no more
information than $Z_{0,\cdots, 0}$.
Moreover, a formula that explicitly expresses  $Z_{N_1,\cdots, N_g}$ in term of $Z_{0,\cdots, 0}$
was also proposed, which seems to be another interesting conjecture that asks for further study.
\end{rem}

\section{Proof of the self-gluing rule for Bogoliubov transforms}\label{sec:proof of self-gluing rule}
The aim of this section is to prove  Theorem \ref{thm:self-gluing rule}.
To make the key idea clear and concrete,
we first consider a simple case of this theorem with full details in \S \ref{subsec:self glu. simple},
and then  sketch the proof of  general cases in \S \ref{subsec:self glu. general}.

\subsection{Simple case}\label{subsec:self glu. simple}
Let
\begin{equation}
V_1 = \exp(\sum_{i,j = 1,2}\sum_{m , n \geq 0}A^{ij}_{mn}\psi^i_\bm\psi^{j*}_{-\bn})|0_{12}\rangle
\end{equation}
be a two-component Bogoliubov transform in $\mathcal{F}_1\otimes \mathcal{F}_2$ and
\begin{equation}
V_2 = \exp(\sum_{m , n \geq 0}A_{mn}\psi^2_\bm\psi^{2*}_{-\bn})|0_{2}\rangle
\end{equation}
be an one-component Bogoliubov transform in $\mathcal{F}_2$, where $A^{ij}_{mn}$ and $A_{mn}$ are arbitrary coefficients maybe with parameters.

Define
\begin{equation}
\tilde{V_2} = \exp(\sum_{m , n \geq 0}Q^{m+n+1}A_{mn}\psi^2_\bm\psi^{2*}_{-\bn})|0_{2}\rangle,
\end{equation}
where $Q$ is a formal variable and $\tilde{V_2}$ is viewed as a formal power series of $Q$ with coefficients in $\mathcal{F}_2$.

 The following inner product
\begin{equation}
\begin{split}
V_0 &= \left(\exp(\sum_{m , n \geq 0}A^{22}_{mn}\psi^2_\bm\psi^{2*}_{-\bn})|0_{2}\rangle ,
 \exp(\sum_{m , n \geq 0}Q^{m+n+1}A_{mn}\psi^2_\bm\psi^{2*}_{-\bn})|0_{2}\rangle\right)\\
    &=\langle 0_2|\exp(\sum_{m , n \geq 0}Q^{m+n+1}A_{mn}\psi^2_{-\bn}\psi^{2*}_{\bm})
     \exp(\sum_{m , n \geq 0}A_{mn}^{22}\psi^2_\bm\psi^{2*}_{-\bn})|0_{2}\rangle
\end{split}
\end{equation}
is well defined as a formal power series of the formal variable $Q$.

Define $\tilde{V} = (V_1 , \tilde{V_2})$,
where the inner product is taken on the $\mathcal{F}_2$ component. Then $\tilde{V}$ is  a formal power series of $Q$ with coefficients in $\mathcal{F}_1$.

\begin{thm}\label{thm:gluing}
Let $\tilde{V}$ and $V_0$ be as above,
then the formal  power series $V = \tilde{V}/V_0$ of $Q$ with coefficients in $\mathcal{F}_1$
is a Bogoliubov transform of the fermionic vacuum $|0_1\rangle \in \mathcal{F}_1$,
 i.e., for  $m , n \geq 0$, there exist  formal power series $R_{mn}$ of $Q$, such that
 \be
 V = \exp(\sum_{m,n\geq 0}R_{mn}\psi^1_\bm\psi^{1*}_{-\bn})|0_{1}\rangle.
 \ee
\end{thm}
From the proof of this theorem, we will see why $V_0$ appear naturally as a factor of $\tilde{V}$.
The following lemma, which is well-known in Lie theory,
will be used in the proof.

\begin{lem} \label{lem:exponential cr}
(See e.g \cite{Hall}) Let $A$ and $B$ be two linear operators on a
vector space $H$. Assume both $e^A$ and $e^B$ make sense. If the
commutator $[A , B] = AB - BA$ commutes with both $A$ and $B$. Then
\be e^A e^B = e^{[A , B]}e^B e^A. \ee
\end{lem}

\begin{proof}
(Proof of Theorem \ref{thm:gluing}) Recall that
\begin{equation*}
\tilde{V}=(\exp(\sum_{i,j = 1,2}\sum_{m , n \geq 0}A^{ij}_{mn}\psi^i_\bm\psi^{j*}_{-\bn})|0_{12}\rangle, \exp(\sum_{m , n \geq 0}Q^{m+n+1}A_{mn}\psi^2_\bm\psi^{2*}_{-\bn})|0_{2}\rangle).
\end{equation*}
To make the notations simpler,
let
\bea
&& \cA^{ij} = \sum_{m , n \geq 0}A^{ij}_{mn}\psi^i_{\bm}\psi^{j*}_{-\bn}, \\
&& \cA = \sum_{m , n \geq 0}Q^{m+n+1}A_{mn} \psi^2_{\bm}\psi^{2*}_{-\bn}.
\eea
Then one can rewrite $\tilde{V}$ as follows:
\be
\tilde{V} = \exp{\cA^{11}} \langle0_2|\exp(\cA^*) \exp (\cA^{21}) \exp (\cA^{12}) \exp (\cA^{22})
|0_2\rangle |0_1\rangle.
\ee
As usual,
our strategy here is to move the annihilators to the right using the
anticommutation relations \eqref{eqn:CR}.
By \eqref{eqn:CR} and \eqref{eqn:sign convention}, one has
\be
[\cA^*, \cA^{21}]
= \cB^{21}=\sum_{m, n \geq 0} B^{21}_{mn} \psi^{2}_{-\bm}\psi^{1*}_{-\bn},
\ee
where
\be
B^{21}_{mn} = \sum_{r \geq 0}Q^{r+m+1}A_{rm}A^{21}_{rn}
\ee
are formal power series of $Q$ which are divisible by $Q$.
Note that the right-hand side  commutes with both
$\cA^*$ and
$\cA^{21}$,
hence by Lemma \ref{lem:exponential cr} we get
\ben
&& \exp(\cA^*) \exp(\cA^{21})
= \exp(\cA^{21}) \exp(\cA^*) \exp(\cB^{21}).
\een
By the same method, one can show that
\be
\exp(\cB^{21})\exp(\cA^{12})
=  \exp(\cA^{2112,1}) \exp(\cA^{12}) \exp(\cB^{21}),
\ee
where
\be
\cA^{2112,1} = [\cB^{21}, \cA^{12}]
= - \sum_{m,n \geq 0} \sum_{r\geq 0} A^{12}_{mr}B^{21}_{rn}\psi^1_m\psi^{1*}_{-n},
\ee
and
\be
\exp(\cB^{21})\exp(\cA^{22})
=  \exp(\cA^{21,1}) \exp(\cA^{22}) \exp(\cB^{21}),
\ee
where
\be
\cA^{21,1} = [\cB^{21}, \cA^{22}]
=  - \sum_{m,n \geq 0}\sum_{r\geq 0} A^{22}_{mr}B^{21}_{rn}\psi^2_m\psi^{1*}_{-n}.
\ee
Now we have
\ben
\tilde{V} = \exp(\cA^{11} + \cA^{2112,1}) \langle 0_2| \exp(\cA^{21}) \exp(\cA^*)  \exp(\cA^{12})\exp(\cA^{22}) \exp(\cA^{21,1}) \exp(\cB^{21})|0_{2}\rangle |0_{1}\rangle.
\een
Note $\cB^{21}|0_2\rangle = 0$ and $\langle 0_2| \cA^{21} = 0$,
we have
\be
\tilde{V} = \exp(\cA^{11} + \cA^{2112,1}) \langle 0_2|
\exp(\cA^*)  \exp(\cA^{12})\exp(\cA^{22}) \exp(\cA^{21,1})
|0_{2}\rangle |0_{1}\rangle.
\ee
Similarly, one can show that
\be
\exp(\cA^*) \exp(\cA^{12})
=  \exp(\cA^{12}) \exp(\cA^*)\exp(\cB^{12}),
\ee
where
\ben
\cB^{12} = [\cA^*, \cA^{12}]
= \sum_{m , n \geq 0}B^{12}_{mn}\psi^1_\bm \psi_\bn^{2*}
\een
which commutes with both $\cA^*$ and $\cA^{12}$, where
\be
B^{12}_{mn} = -\sum_{r\geq 0} A^{12}_{mr}Q^{n+r+1}A_{nr}
\ee
are formal power series of $Q$ which are divisible by $Q$.

\be
\exp(\cB^{12}) \exp(\cA^{22}) = \exp(\cA^{22}) \exp (\cA^{12,1}) \exp(\cB^{12}),
\ee
where
\be
\cA^{12,1} = [\cB^{12}, \cA^{22}]
= \sum_{m,n \geq 0} \sum_{r \geq 0} B^{12}_{mr} A^{22}_{rn} \psi_\bm^1 \psi^{2*}_{-\bn},
\ee
and
\be
\exp(\cB^{12}) \exp(\cA^{21,1})
= \exp(\cA^{1221,1})\exp(\cA^{21,1})\exp(\cB^{12}),
\ee
where
\be
\cA^{1221,1} = [\cB^{12}, \cA^{21,1}] = \sum_{m,n\geq 0}( \sum_{r\geq 0}B^{12}_{mr}A^{21,1}_{rn})\psi^1_{\bm}\psi^{1*}_{-\bn}
\ee
commutes with both $\cB^{12}$ and $\cA^{21,1}$.
Because $\langle 0_2| \cA^{12} = 0$ and $\cB^{12}|0_2\rangle=0$,
\ben
&&  \langle 0_2| \exp(\cA^*)  \exp(\cA^{12})\exp(\cA^{22}) \exp(\cA^{21,1}) |0_{2}\rangle  \\
& = &  \langle 0_2| \exp(\cA^{12})\exp(\cA^*)  \exp(\cB^{12})
\exp(\cA^{22}) \exp(\cA^{21,1})|0_{2}\rangle  \\
&  = & \langle 0_2|
\exp(\cA^*)  \exp(\cA^{22}) \exp(\cA^{12,1}) \exp(\cA^{1221,1})\exp(\cA^{21,1})\exp(\cB^{12})
|0_{2}\rangle  \\
&  = & \exp(\cA^{1221,1}) \langle 0_2|
\exp(\cA^*)  \exp(\cA^{22}) \exp(\cA^{12,1})\exp(\cA^{21,1})
|0_{2}\rangle.
\een
Because the operators $\cA^{22}, \cA^{12,1}$ and $\cA^{21,1}$ commute with each other,
we now have
\ben
\tilde{V} =
\exp(\cA^{11}+\cA^{2112,1}+\cA^{1221,1})
\langle 0_2|\exp(\cA^{*})\exp(\cA^{21,1})\exp(\cA^{12,1})\exp(\cA^{22}) |0_2\rangle|0_1\rangle.
\een
Recall $\cA^{2122,1}, \cA^{12,1}, \cA^{21,1}$ are divisible by $Q$,
and $\cA^{1221,1}$ is divisible by $Q^2$.
By repeating the above procedure $N$-times one gets:
\ben
\tilde{V} & = &
\exp(\cA^{11}+\sum_{j=1}^N (\cA^{2112,j}+\cA^{1221,j}) \\
&& \langle 0_2|\exp(\cA^{*})\exp(\cA^{21,N})\exp(\cA^{12,N})\exp(\cA^{22})) |0_2\rangle|0_1\rangle,
\een
where $\cA^{2121,j}, \cA^{12,j}, \cA^{21,j}$
and $\cA^{1221,j}$ is divisible by $Q^j$.
Therefore,
by taking $N \to \infty$,
\be
\tilde{V} = \langle 0_2|\exp(\cA^{*}) \exp(\cA^{22}) |0_2\rangle \cdot
\exp(\cA^{11}+\sum_{j=1}^\infty (\cA^{2112,j}+\cA^{1221,j}) |0_1\rangle.
\ee
This completes the proof of Theorem 4.1.
\end{proof}

\subsection{Sketch of the proof of Theorem \ref{thm:self-gluing rule}}\label{subsec:self glu. general}
The proof of Theorem \ref{thm:self-gluing rule} is essentially same as
that of Theorem \ref{thm:gluing}, but with much more complicated calculations.
So we just sketch it here.

\begin{proof}$(Sketch\ of \ the\ Proof\ of\ theorem\ \ref{thm:self-gluing rule})$
Recall that the gluing vector $P^\bE_{ab}$ is defined in \eqref{eqn:the gluing operator}. We write
$$V = \exp\left(\sum_{i,j\in\{a, b, 1, 2, \cdots , M\}}\sum_{m , n \geq 0}A^{ij}_{mn}\psi^i_\bm\psi^{j*}_{-\bn}\right)|0\rangle.$$
The self-gluing $\tilde{G}^\bE(V)$ is given by \eqref{eqn:self-gluing}. For $r\in\mathbb{Z}+1/2$, the adjoint operator of $\psi^{i}_{r}$ is
$\psi^{i*}_{r}$. Let
$$\cA_{ij}=\sum_{m,n\geq 0}A^{ij}_{mn}\psi^i_\bm\psi^{j*}_{-\bn},\ \ i, j\in\{1, \cdots, n, a, b\},$$
$$\cB_{ij}=\sum_{m,n\geq 0}B^{ij}_{mn}\psi^i_\bm\psi^{j*}_{-\bn},\  \ \ i, j = a, b,$$
where $B^{ij}_{mn} = Q^{m+n+1}\Theta^{\delta^{ab}_{ij}}E^{ij}_{mn}$ for $i , j = a , b$. Then
\be
\begin{split}\label{eqn:self-gluing expansion}
\tilde{G}^\bE(V)=&\exp(\sum_{i,j=1}^M \cA_{ij})
             \langle 0|\exp(\cB^*_{bb})\exp(\cB^*_{ba})\exp(\cB^*_{ab})\\
             &\exp(\cB^*_{aa})\exp(\cA_{aa})\exp(\cA_{ab})\exp(\cA_{ba})\exp(\cA_{bb})\\
             & \prod_{i=1}^M \exp(\cA_{ai})\exp(\cA_{bi})
              \prod_{i=1}^M \exp(\cA_{ia})\exp(\cA_{ib})|0\rangle|0'\rangle,
\end{split}
\ee
where $|0\rangle$ is the vacuum in $\cF_a\otimes\cF_b$ and $|0'\rangle$ is the vacuum  in
 $\cF_1\otimes\cdots\otimes\cF_M$.

Note that $$\langle 0|\exp(\cA_{ip})=\langle 0|\exp(\cA_{pi})=\langle 0|$$ for $i= 1, \cdots, M,\ p=a,b,$
so we try to move these operators to the left.
We first consider $\exp(\cA_{ia})$.
By \eqref{eqn:CR} and \eqref{eqn:sign convention},
for a fixed $i$, $1\leq i \leq M$, we have
\be
[\cB^*_{aa} , \cA_{ia}] = \cS_{ia}=\sum_{m,n\geq 0}S^{ia}_{mn}\psi^i_\bm\psi^{a*}_\bn,
\ee
where
$$S^{ia}_{mn} = \sum_{r\geq 0}A^{ia}_{mr}B^{aa}_{nr}$$
are formal power series of $Q$, $\Theta$ and $\Theta^{-1}$.
It is important for our argument that the coefficients $S_{mn}$
are divisible by $Q$. It is clear that $\cS_{ia}$ commutes with
both $B^*_{aa}$ and $\cA_{ia}$.
By  Lemma \ref{lem:exponential cr}, we have
\be
\exp(\cB^*_{aa})\exp(\cA_{ia})) = \exp(\cA_{ia})\exp(\cB^*_{aa})\exp(\cS_{ia})).
\ee

Similarly, one can prove that
\be
\exp(\cB^*_{ba})\exp(\cA_{ia})) = \exp(\cA_{ia})\exp(\cB^*_{ba})\exp(\cS_{ib})),
\ee
where $$\cS_{ib}=\sum_{m,n\geq 0}S^{ib}_{mn}\psi^i_\bm\psi^{b*}_\bn$$ is an operator divisible by $Q$.

Note that $\langle 0|\psi^{a*}_{-\bn} = 0$ for all $n\geq 0$, we have
\be\label{eqn:step1 cr}
\begin{split}
 &\langle 0|\exp(\cB^*_{bb})\exp(\cB^*_{ba})\exp(\cB^*_{ab})\exp(\cB^*_{aa})\exp(\cA_{ia})\\
=&\langle 0|\exp(\cB^*_{bb})\exp(\cB^*_{ba})\exp(\cB^*_{ab})\exp(\cB^*_{aa})\exp(\cS_{ib})\exp(\cS_{ia}).
\end{split}
\ee
Now we get two new operators $\exp{\cS_{ia}}$ and $\exp{\cS_{ib}}$
that seems unrelated to the original representation of $\tilde{G}^E(V)$ in \eqref{eqn:self-gluing expansion}.
Note that $\exp{\cS_{ia}}|0\rangle=\exp{\cS_{ib}}|0\rangle=|0\rangle$,
we can get ride of these operators by moving them to the right.
By the same argument, one can show that
\be\label{eqn:step2 cr}
\begin{split}
 &\exp(\cS_{ib})\exp(\cS_{ia})\exp(\cA_{aa})\exp(\cA_{ab})\exp(\cA_{ba})\exp(\cA_{bb})\\
=&\exp(\cA_{aa})\exp(\cA_{ab})\exp(\cA_{ba})\exp(\cA_{bb})
  \exp(\cA'_{ia,1})\exp(\cA'_{ib,1})\exp(\cS_{ib})\exp(\cS_{ia}),
\end{split}
\ee
where
$$\cA'_{ia,1}=\sum_{mn\geq 0}A'^{ia,1}_{mn}\psi^i_\bm\psi^{a*}_{-\bn},\ \cA'_{ib,1}=\sum_{m,n\geq 0}A'^{ib,1}_{mn}\psi^i_\bm\psi^{b*}_{-\bn}$$
are certain operators divisible by $Q$. Finally we have the commutation relation
\be\label{eqn:step3 cr}
\begin{split}
 &\exp(\cS_{ib})\exp(\cS_{ia})\prod_{j=1}^M\exp(\cA_{aj})\exp(\cA_{bj})\\
=&\exp(\sum_{j=1}^M \cA'_{ij,1})\prod_{j=1}^M\exp(\cA_{aj})\exp(\cA_{bj})\exp(\cS_{ib})\exp(\cS_{ia}),
\end{split}
\ee
where, for $j= 1, \cdots, M$,
$$\cA'_{ij,1}=\sum_{m,n\geq 0}A'^{ij,1}_{mn}\psi^i_\bm\psi^{j*}_{-\bn}$$
are certain operators divisible by $Q$. By \eqref{eqn:step1 cr}, \eqref{eqn:step2 cr} , \eqref{eqn:step3 cr} and \eqref{eqn:self-gluing expansion}, we get
\be
\begin{split}
\tilde{G}^\bE(V)=&\exp(\sum_{i,j=1}^M(\cA_{ij}+\cA'_{ij,1}))
                 \langle 0|\exp(\cB^*_{bb})\exp(\cB^*_{ba})\exp(\cB^*_{ab})\exp(\cB^*_{aa})\\
                 &\exp(\cA_{aa})\exp(\cA_{ab})\exp(\cA_{ba})\exp(\cA_{bb})\prod_{i=1}^M \exp(\cA_{ai})\exp(\cA_{bi})\\
                 &\prod_{i=1}^M \exp(\cA_{ib})\exp(\sum_{i=1}^M\cA'_{ia,1})\exp(\sum_{i=1}^M\cA'_{ib,1})|0\rangle|0'\rangle.
\end{split}
\ee

Repeating the above process for the operator $\prod_{i=1}^M \exp(\cA_{ib})$,
we see that, for $i,j=1, \cdots, M$, there are operators $$\cA'_{ij,2}=\sum_{m,n\geq 0}A'^{ij,2}_{mn}\psi^i_\bm\psi^{j*}_{-\bn},$$
 $$\cA'_{ia,2}=\sum_{m,n\geq 0}A'^{ia,2}_{mn}\psi^i_\bm\psi^{a*}_{-\bn},$$
 $$\cA'_{ib,2}=\sum_{m,n\geq 0}A'^{ib,2}_{mn}\psi^i_\bm\psi^{b*}_{-\bn}$$
  divisible by $Q$ such that

\be\label{eqn:self-gluing expansion 2}
\begin{split}
&\tilde{G}^\bE(V)\\
=&\exp(\sum_{i,j=1}^M(\cA_{ij}+\cA'_{ij}))
                 \langle 0|\exp(\cB^*_{bb})\exp(\cB^*_{ba})\exp(\cB^*_{ab})\\
                 &\exp(\cB^*_{aa})\exp(\cA_{aa})\exp(\cA_{ab})\exp(\cA_{ba})\exp(\cA_{bb})\prod_{i=1}^M \exp(\cA_{ai})\exp(\cA_{bi})\\
                 &\exp(\sum_{i=1}^M\cA_{ia}(1))\exp(\sum_{i=1}^M\cA_{ib}(1))|0\rangle|0'\rangle,
\end{split}
\ee
where $\cA'_{ij}=\cA'_{ij,1}+\cA'_{ij,2}$, $\cA_{ia}(1)=\cA'_{ia,1}+\cA'_{ia,2}$, $\cA_{ib}(1)=\cA'_{ib,1}+\cA'_{ib,2}$.

We can carry out similar process for the operators $\prod_{i=1}^M \exp(\cA_{ai})$ and $\prod_{i=1}^M \exp(\cA_{bi})$
to show that, for $i=1, \cdots, M$,  there are operators
$$\cA''_{ij}=\sum_{m,n\geq 0}A''^{ij}_{mn}\psi^i_\bm\psi^{j*}_{-\bn},$$
$$\cA_{ai}(1)=\sum_{m,n\geq 0}A^{ai}_{mn}(1)\psi^a_\bm\psi^{i*}_{-\bn},$$
$$\cA_{bi}(1)=\sum_{m,n\geq 0}A^{bi}_{mn}(1)\psi^a_\bm\psi^{i*}_{-\bn}$$
divisible by $Q$ such that

\be
\begin{split}\label{eqn:self-gluing expansion 3}
\tilde{G}^\bE(V)=&\exp(\sum_{i,j=1}^M (\cA_{ij}+\cA_{ij}(1)))
             \langle 0|\exp(\cB^*_{bb})\exp(\cB^*_{ba})\exp(\cB^*_{ab})\\
             &\exp(\cB^*_{aa})\exp(\cA_{aa})\exp(\cA_{ab})\exp(\cA_{ba})\exp(\cA_{bb})\\
             & \prod_{i=1}^M \exp(\cA_{ai}(1))\exp(\cA_{bi}(1))
              \prod_{i=1}^M \exp(\cA_{ia}(1))\exp(\cA_{ib}(1))|0\rangle|0'\rangle,
\end{split}
\ee
where $\cA_{ij}(1)=\cA'_{ij}+\cA''_{ij}.$

At the first glance, it seems that we go back to the starting point \eqref{eqn:self-gluing expansion}
and get nothing helpful.
But it is not the case.
The key point here is that all the operators $\cA_{ij}(1), \cA_{ia}(1), \cA_{ib}(1), \cA_{ai}(1)$ and $\cA_{bi}(1)$,
when viewed as formal power series of $Q$, are divisible by $Q$.
Repeat the above process inductively, for $N = 1 , 2 , \cdots$, we get a series of operators
$\cA_{ij}(N), \cA_{ia}(N), \cA_{ib}(N), \cA_{ai}(N)$ and $\cA_{bi}(N)$
which are formal power series of $Q$, $\Theta$ and  $\Theta^{-1}$ and divisible
by $Q^N$, such that
\be\label{eqn:self-gluing expansion key}
\begin{split}
\tilde{G}^\bE(V)=&\exp(\sum_{i,j=1}^M (\sum_{k=0}^{N+1}\cA_{ij}(k)))
             \langle 0|\exp(\cB^*_{bb})\exp(\cB^*_{ba})\exp(\cB^*_{ab})\\
             &\exp(\cB^*_{aa})\exp(\cA_{aa})\exp(\cA_{ab})\exp(\cA_{ba})\exp(\cA_{bb})\\
             & \prod_{i=1}^M \exp(\cA_{ai}(N+1))\exp(\cA_{bi}(N+1))\\
              &\prod_{i=1}^M \exp(\cA_{ia}(N+1))\exp(\cA_{ib}(N+1))|0\rangle|0'\rangle,
\end{split}
\ee
for each integer $N>0$, where we set $\cA_{ij}(0) = \cA_{ij}$.

For a formal power series $f = \sum_{n\geq 0}f_n Q^n$ of $Q$ with coefficients in $\cF_1\otimes\cdots\otimes\cF_M$ and a positive integer $N$, we denote by $[f]_N$ the sum $\sum_{n=0}^N f_n Q^n$. Note that by definition
\be
\begin{split}
\tilde{G}^\bE(V)_{closed}=&
             \langle 0|\exp(\cB^*_{bb})\exp(\cB^*_{ba})\exp(\cB^*_{ab})\exp(\cB^*_{aa})\\
             &\exp(\cA_{aa})\exp(\cA_{ab})\exp(\cA_{ba})\exp(\cA_{bb})
             |0\rangle.
\end{split}
\ee
By \eqref{eqn:self-gluing expansion key}, for each positive integer $N$, we have
\be\label{eqn:N-cut indentity}
[\tilde{G}^\bE(V)]_N
=\left[\tilde{G}^\bE(V)_{closed} \exp(\sum_{i,j=1}^M(\sum_{k=0}^{N+1}\cA_{ij}(k)))
|0'\rangle\right]_N.
\ee
For $i , j = 1 , \cdots , M$, we define
\be
\mathcal{R}_{ij} = \sum_{k\geq 0}\cA_{ij}(k).
\ee
Since $\cA_{ij}(k)$ are divisible by $Q^k$, $k>0$, $\mathcal R_{ij}$ are well defined as formal power series of $Q$
and have the form
$$\mathcal R_{ij}=\sum_{m,n\geq 0}R_{mn}^{ij}\psi^i_\bm\psi^{j*}_{-\bn}.$$
By \eqref{eqn:N-cut indentity}, we have
\be
[\tilde{G}^\bE(V)]_N = \left[\tilde{G}^\bE(V)_{closed} \exp(\sum_{i,j=1}^M \mathcal R_{ij})
|0'\rangle\right]_N
\ee
for all positive integer $N$. So we have
\be
\tilde{G}^\bE(V) = \tilde{G}^\bE(V)_{closed} \exp(\sum_{i,j=1}^M\sum_{m,n\geq 0}R^{ij}_{mn}\psi^i_\bm\psi^{j*}_{-\bn})
|0'\rangle,
\ee
and hence
\be
G^\bE(V) = \exp(\sum_{i,j=1}^M\sum_{m,n\geq 0}R^{ij}_{mn}\psi^i_\bm\psi^{j*}_{-\bn})
|0'\rangle
\ee
is a Bogoliubov transform.
\end{proof}

\maketitle

\end{document}